\documentclass[10pt,a4paper]{amsart}
\usepackage[utf8]{inputenc}
\usepackage{xcolor}
\usepackage[margin=3cm]{geometry}
\usepackage[pdftex,hyperfootnotes=false,colorlinks=true,linkcolor=blue,
citecolor=purple,filecolor=magenta,urlcolor=cyan,hypertexnames=false]{hyperref}
\usepackage{mathabx}

\usepackage[noabbrev]{cleveref}
\expandafter\def\csname ver@etex.sty\endcsname{3000/12/31}

\crefname{lemma}{lemma}{lemmata}
\Crefname{lemma}{Lemma}{Lemmata}
\crefname{subsection}{subsection}{subsections}
\Crefname{subsection}{Subsection}{Subsections}

\title[Cut-and-join equation for monotone Hurwitz numbers revisited]{Cut-and-join equation for monotone Hurwitz numbers revisited}

\author[P.~Dunin-Barkowski]{P.~Dunin-Barkowski}

\address[P.~Dunin-Barkowski]{Faculty of Mathematics, National Research University Higher School of Economics, Usacheva 6, 119048 Moscow, Russia; and ITEP, 117218 Moscow, Russia}
\email{ptdunin@hse.ru}

\author[R.~Kramer]{R.~Kramer}

\address[R.~Kramer]{Korteweg-de Vriesinstituut voor Wiskunde, 
	Universiteit van Amsterdam, Postbus 94248,
	1090GE Amsterdam, Nederland}
\email{r.kramer@uva.nl}

\author[A.~Popolitov]{A.~Popolitov}

\address[A.~Popolitov]{Department of Physics and Astronomy, Uppsala University, Uppsala, Sweden; Institute for Information Transmission Problems, Moscow 127994, Russia; and ITEP, Moscow 117218, Russia}
\email{popolit@gmail.com}

\author[S.~Shadrin]{S.~Shadrin}

\address[S. Shadrin]{Korteweg-de Vriesinstituut voor Wiskunde, 
	Universiteit van Amsterdam, Postbus 94248,
	1090GE Amsterdam, Nederland}
\email{s.shadrin@uva.nl}

\newcommand{\cont}{\mathord{\textup{cr}}}

\newcommand{\cor}[1]{\big{\langle} 0 \big|\, #1 \, \big| 0 \big{\rangle}}

\DeclareMathOperator{\Res}{Res}

\newtheorem{theorem}{Theorem}[section]

\newtheorem{lemma}[theorem]{Lemma}
\theoremstyle{definition}

\newtheorem{remark}[theorem]{Remark}

\begin{document}

\begin{abstract}
We give a new proof of the cut-and-join equation for the monotone Hurwitz numbers, derived first by Goulden, Guay-Paquet, and Novak.  
The main interest in this particular equation is its close relation to the quadratic loop equation in the theory of spectral curve topological recursion, and we recall this motivation giving a new proof of the topological recursion for monotone Hurwitz numbers, obtained first by Do, Dyer, and Mathews.
\end{abstract}

\maketitle

\tableofcontents

\section{Introduction}

The monotone Hurwitz numbers were introduced by Goulden, Guay-Paquet, and Novak in relation to the HCIZ integral~\cite{GouldenHCIZ}, and they discovered many properties of these numbers that generalized known results for the usual Hurwitz numbers. In particular, they conjectured in~\cite{GouldenetalGenus0} that monotone Hurwitz numbers satisfy the spectral curve topological recursion~\cite{EynardOrantin}, and this conjecture was proved by Do, Dyer, and Mathews in~\cite{DoDyerMathews}. 

These numbers are related to hypergeometric tau-functions~\cite{HarnadOrlov}, and there is a number of new results and conjectures about them and their orbifold generalization, see~\cite{ACEH,ALS,DoKarev,KLS}. In particular, Do and Karev conjectured in~\cite{DoKarev} that orbifold monotone Hurwitz numbers also satisfy topological recursion, and provided the spectral curve data for it, but, to the best of our knowledge, their conjecture is still open (though some progress was achieved in~\cite{ACEH,KLS}). 

This goal of this paper is  to understand better the connection between the  operators $B_b^{\leq}$, $b\geq 1$, that produce monotone Hurwitz numbers in terms of the representation theory of the symmetric group~\cite{ALS}, and the form of the cut-and-join equation obtained by Goulden, Guay-Paquet, and Novak.  The eigenvalue of the operator $B_b^{\leq}$ on an irreducible representation indexed by a partition $\lambda$ is given by a complete homogeneous symmetric polynomial of degree $b$ of its content vector.

In the meanwhile, the cut-and-join equation~\cite[Theorem 1.2]{GouldenetalGenus0} seems to reflect the  exponential action of the second Casimir operator $C_2$, whose eigenvalue on an irreducible representation indexed by partition $\lambda=(\lambda_1\geq \cdots \geq \lambda_\ell>0)$ is given by 
\begin{equation*}
\frac 12 \sum_{i=1}^\ell \left[
\left( \lambda_i-i+\frac12 \right)^2
- \left( -i+\frac12 \right)^2
\right],
\end{equation*}
which is equal to the sum of the components of the content vector. There is no obvious way to connect $B_b^\geq$, $b\geq 1$, and $C_2$, and the goal of this paper is to give a new proof of the theorem of Goulden, Guay-Paquet, and Novak in such a way that this subtle point would be clarified.

Note that Goulden-Guay-Paquet-Novak's proof involves careful analysis of what happens with all the permutations under the ``cut" and two different ``join" operations, while our proof is completely different and is given in terms of the operators on the semi-infinite wedge (or, in other words, fermionic) space. One more alternative proof that works, however, only in genus $0$ was recently found by Carrell and Goulden~\cite{CarrellGoulden}.


Our main interest in analysing in detail how the cut-and-join operator of Goulden, Guay-Paquet, and Novak occurs in the theory of monotone Hurwitz numbers comes from the prominent role that this explicit operator plays in the theory of topological recursion. The technique developed in~\cite{BorotEynardOrantin,BorotShadrin} allows us to immediately derive the topological recursion statement from this particular shape of the cut-and-join operator, once we know the (quasi-)polynomiality property of monotone Hurwitz numbers.

The required polynomiality property was first proved in~\cite[Theorem 1.5]{GouldenPolynomiality}, and an alternative proof is available as a special case of~\cite[Theorem 5.2]{KLS}. So, applying these results, we immediately obtain a new proof of topological recursion for monotone Hurwitz numbers, first proved in~\cite[Theorem 1]{DoDyerMathews}.

We expect that the approach described in the present paper might be helpful in developing a proof of the aforementioned open Do-Karev conjecture on the topological recursion for the orbifold monotone case.

\subsection{Structure of the paper}
In \cref{sec:CJ} we recall one of the possible forms of the cut-and-join equation of Goulden, Guay-Paquet, and Novak for the monotone Hurwitz numbers and give a new proof for it.
In \cref{sec:TR} we recall the topological recursion statement of Do, Dyer, and Mathews, and give a new proof for it. 

\subsection{Acknowledgments}
P.D.-B. was supported by RFBR grant 16-31-60044-mol\_a\_dk.
S.S. and R.K. were supported by the Netherlands Organization for Scientific Research.
A.P. was supported by the
the grant ``Geometry and Physics'' from the Knut and Alice Wallenberg foundation%
. 

We thank Guoniu Han and Huan Xiong for useful remarks. 

\section{Monotone Hurwitz numbers and cut-and-join equation}\label{sec:CJ}

For a partition $\mu\vdash d$ we denote by $\ell(\mu)$ its length and by $\mu_1\geq \dots \geq \mu_{\ell(\mu)}$ its parts. We associate to $\mu$ a Young diagram that we denote also by $\mu$, whose boxes have row coordinates from $1$ to $\ell(\mu)$ and the column coordinates in the $i$-th row from $1$ to $\mu_i$. By $|\mu|$ we denote the number of boxes, that is $|\mu|:=d$.

A disconnected monotone Hurwitz number depends on a partition $\mu\vdash d$ of a positive integer $d$ (the degree) and on a genus $g\in \mathbb{Z}$ that can be, potentially, negative (since we have the disconnected case), but the parameter $b=2g-2+d+\ell(\mu)$ must be non-negative. It is given by 
\begin{equation*}
H^\bullet_{g,\mu} := \sum_{\lambda\vdash d} \frac{\dim \lambda}{d!} \frac{\chi_\lambda(\mu) } {\prod_{i=1}^{\ell(\mu)} \mu_i} h_{2g-2+d+\ell(\mu)} (\cont^\lambda_1,\dots,\cont^\lambda_d).
\end{equation*}
Here $h_{2g-2+d+\ell(\mu)}=h_b$ is the full homogeneous symmetric function of its variables of degree $b=2g-2+d+\ell(\mu)$, and $\cont^\lambda$ is the vector of contents of the standard Young tableau associated to the diagram $\lambda$ that is, if $c_i$ and $r_i$ are the column and the row indices of the box $i$, then $\cont^\lambda_i:=c_i-r_i$.
By $\dim \lambda$ we denote the dimension of the representation $\lambda$, and by $\chi_\lambda(\mu)$ the character of the representation $\lambda$ evaluated at the the conjugacy class of cycle type \( \mu \).

In order to relate this definition to the notation used in the introduction let us mention that $h_{b} (\cont^\lambda_1,\dots, \cont^\lambda_n)$, $b\geq 0$, can be considered as the eigenvalue of the action of the operator $B_b^\geq$ in the representation $\lambda$, where $B_b^\geq$ is defined as the central element given by the full homogeneous symmetric function of degree $b$ of the Jucys-Murphy elements $\mathcal{J}_2,\dots,\mathcal{J}_{|\lambda|}$, see details in~\cite{ALS}.


The generating function for the disconnected monotone Hurwitz numbers is defined as
\begin{equation*}
\mathcal{Z}(s,t,\mathbf{p}) := 1+\sum_{\substack{d\geq 1,\ \mu\vdash d,\ b\geq 0 \\ g:=\frac{b-d-\ell(\mu))}2 \in \mathbb{Z}
	}
} \frac{s^d}{d!} t^{b}  H^\bullet_{g,\mu} p_{\mu_1}\cdots p_{\mu_{\ell(\mu)}}
\end{equation*}
We give a new proof of the following cut-and-join equation for this function, which was first derived in~\cite[Theorem 1.2]{GouldenetalGenus0}.

\begin{theorem} 
We have:
	\begin{equation}\label{eq:CJ}
	\frac{1}{2t} \left[s\frac{\partial}{\partial s} - sp_1\right] \mathcal{Z}= 
	\frac 12
	\sum_{i,j=1}^\infty 
	\left[
	(i+j)p_ip_j \frac{\partial} {\partial p_{i+j}} + ij p_{i+j} \frac{\partial^2}{\partial p_i\partial p_j}
	\right]
	\mathcal{Z}.
	\end{equation}
\end{theorem}

\begin{proof}
  It is convenient to rewrite \cref{eq:CJ} using the semi-infinite wedge formalism interpretation of the monotone Hurwitz developed in~\cite{ALS}. We do not repeat here the definition of the semi-infinite wedge space and the relevant operators,
  as it can be found in~\cite[Section 2]{KLS}.
  Moreover, in physics literature semi-infinite wedge formalism is widely known under the name of
  ``free-fermion formalism'' and is a standard tool in theory of KP/Toda integrability of matrix models
  ~\cite[Section 4.2]{MorMMIntS}.
  In the standard notation we have:
\begin{equation*}
\mathcal{Z}(s,t,\mathbf{p}) =  \cor{ e^{\sum_{i=1}^\infty \frac{\alpha_ip_i}{i}} \mathcal{D}^{(h)}(t) e^{\alpha_{-1}s} } ,
\end{equation*}
where the operator $\mathcal{D}^{(h)}(t)$ acts diagonally on the basis vectors
\begin{equation*}
v_\lambda\coloneq \underline{(\lambda_1-\frac 12)}\wedge \underline{(\lambda_2-\frac 32)}\wedge \underline{(\lambda_3-\frac 52)}\wedge \dots
\end{equation*}
with the eigenvalue given by the action of $\sum_{b=0}^\infty B^\geq_b t^b$ in the representation $\lambda$.

\Cref{eq:CJ} is equivalent to the following one:
\begin{align} \label{eq:CJ-cor}
& \frac s{2t}  \left[ \cor{ e^{\sum_{i=1}^\infty \frac{\alpha_ip_i}{i}} \mathcal{D}^{(h)}(t) \alpha_{-1} e^{\alpha_{-1}s} } -
\cor{ e^{\sum_{i=1}^\infty \frac{\alpha_ip_i}{i}} \alpha_{-1} \mathcal{D}^{(h)}(t) e^{\alpha_{-1}s} }
\right]
\\ \notag &
= \cor{ e^{\sum_{i=1}^\infty \frac{\alpha_ip_i}{i}} \mathcal{D}^{(h)}(t) \mathcal{F}_2 e^{\alpha_{-1}s} }.
\end{align}
In order to understand the left hand side of this equation, we have to compute $[\mathcal{D}^{(h)}(t) \alpha_{-1}] e^{\alpha_{-1}s}  \,\big|0\big>$. By definition, $\mathcal{D}^{(h)}(t)v_\lambda = \prod_{i=1}^{|\lambda|} (1-t\cdot cr_\lambda^i)^{-1}$. For any $\lambda\vdash d$ we denote by $\lambda\setminus 1$ set of partitions $\nu\vdash d-1$ whose Young diagram can be obtained by removing one corner box from the Young diagram of $\lambda$.  We denote this special corner box by $\Box_{\lambda/\nu}$. Recall that $\dim\lambda = \sum_{\nu\in\lambda\setminus 1} \dim \nu$, by the Murnaghan-Nakayama rule. The action of $\alpha_{-1}$ is given by $\alpha_{-1}v_\lambda = \sum_{\nu\setminus 1 \ni \lambda} v_\nu$. Using these formulas, we have:
\begin{align*}
 [\mathcal{D}^{(h)}(t),\alpha_{-1}]e^{s\alpha_{-1}} \,\big|0\big> 
& = \sum_{\lambda} \left[\prod_{\Box \in \lambda} \frac 1{1 - t \, \cont^\lambda_\Box} \right] s^{|\lambda|} \frac {\dim\lambda} {|\lambda|!} \sum_{\nu \setminus 1 \ni \lambda} \frac{
	t\, \cont^\nu_{\Box_{\nu/\lambda}}
}{
1-t\, \cont^\nu_{\Box_{\nu/\lambda}}
} v_\nu
\\
& = t\sum_{\nu} \left[\prod_{\Box \in \nu} \frac 1{1-t \, \cont^\nu_\Box} \right] \frac{s^{|\nu|-1}}{(|\nu|-1)!} \sum_{\lambda\in\nu\setminus 1} \cont^\nu_{\Box_{\nu/\lambda}} \dim\lambda\, v_\nu.
\end{align*}
On the right hand side we recall that $\mathcal{F}_2v_\lambda = C_2(\lambda)v_\lambda$, and we have 
\begin{equation*}
 \mathcal{D}^{(h)}(t)\mathcal{F}_2e^{s\alpha_{-1}} \,\big|0\big> = \sum_{\nu} \left[\prod_{\Box \in \nu} \frac 1{1-t \, \cont^\nu_\Box} \right]
 \frac{s^{|\nu|}}{|\nu|!} \dim\nu \left[\sum_{\Box \in\nu} \cont^\nu_\Box \right] v_\nu.
\end{equation*}
So, \cref{eq:CJ-cor} is equivalent to the following statement that should be true for any Young diagram $\nu$:
\begin{equation*}
|\nu| \sum_{\lambda\in\nu\setminus 1} \dim\lambda\cdot \cont^\nu_{\Box_{\nu/\lambda}} = 2 \dim\nu \sum_{\Box \in \nu} \cont^\nu_\Box
\end{equation*}
Recall that $\dim\nu = |\nu|!/H_{\nu}$, where by $H_\nu$ we denote the product of the hook lengths of all boxes in $\nu$. Thus we reduce \cref{eq:CJ-cor} to the following equation for any Young diagram $\nu$:
\begin{equation*}
\sum_{\lambda\in\nu\setminus 1} \frac{cr^\nu_{\Box_{\nu/\lambda}}}{H_\lambda} = \frac{2}{H_\nu} \sum_{\Box\in\nu} cr^\nu_\Box. 
\end{equation*}
We prove this equation below, see Lemma~\ref{lem:combidentity}, and this completes the proof of the theorem.
\end{proof}

\begin{remark} Note how amazing and combinatorially non-trivial  the occurrence of the operator $\mathcal{F}_2$ is in the case of the monotone Hurwitz numbers. Compare it with  the case of the usual Hurwitz numbers, where the generating function is given by $ \cor{ e^{\sum_{i=1}^\infty \frac{\alpha_ip_i}{i}} e^{t\mathcal{F}_2} e^{\alpha_{-1}} } $ and the standard cut-and-join equation
\begin{equation*}
\frac{\partial}{\partial t}\cor{ e^{\sum_{i=1}^\infty \frac{\alpha_ip_i}{i}} e^{t\mathcal{F}_2} e^{\alpha_{-1}} } 
= \cor{ e^{\sum_{i=1}^\infty \frac{\alpha_ip_i}{i}} \mathcal{F}_2 e^{t\mathcal{F}_2} e^{\alpha_{-1}} } 
\end{equation*}	
is in this form natural and does not require any further computation. 	
\end{remark}

\begin{lemma} \label{lem:combidentity} 
For any Young diagram $\nu$ we have: 
\begin{equation*} 
\sum_{\lambda\in\nu\setminus 1} \frac{\cont^\nu_{\Box_{\nu/\lambda}}}{H_\lambda} = \frac{2}{H_\nu} \sum_{\Box\in\nu} \cont^\nu_\Box. 
\end{equation*}
\end{lemma}

This lemma might be obvious for the experts in the representation theory of the symmetric group, as it can be derived from the Vershik-Okounkov approach~\cite{VershikOkounkov} (we thank experts for pointing this out). However, with a view towards its generalizations for symmetric functions of the contents of skew hooks (that would be necessary if one tries to generalize the same approach that we use here to the orbifold case), we feel that the recent papers by Dehaye, Han, and Xiong~\cite{HanShifted,Dehaye,HanXiong} provide the most suitable combinatorial tools. So, we also give an alternative proof of this lemma that derives it from a remarkable result of Han on the so-called $g$-functions of  integer partitions~\cite[Theorem 1.1]{HanShifted}. Note that this theorem of Han has already been applied in the theory of topological recursion in a completely different context, see~\cite{DbMNPS}.

\begin{proof}[Proof using the action of Jucys-Murphy elements in representations of symmetric groups] 
Recall that the elements of the Gelfand-Tsetlin (Young) basis of the representation $V^\nu$ of the symmetric group $S_{|\nu|}$ are identified with Young tableaux associated with $\nu$. Recall also that the action of the vector of the Jucys-Murphy elements $(\mathcal{J}_1=0,\mathcal{J}_2,\dots,\mathcal{J}_{|\nu|})$ in the Gelfand-Tsetlin basis is diagonal, and the vector of eigenvalues on a particular Young tableau is given by its vector of contents~\cite{VershikOkounkov}.
	
Consider the operator $\mathcal{J}_{|\nu|}$ acting on $V^{\nu}$. We compute its trace in two different ways. 

On the one hand, it commutes with the action of $S_{|\nu|-1}$. Therefore it acts by scalar multiplication on the irreducible factors of $V^{\nu}$ considered as a representation of $S_{|\nu|-1}$. Since $V^{\nu}$ splits as $\oplus_{\lambda\in\nu\setminus 1} V^{\lambda}$ and the action of $\mathcal{J}_{|\nu|}$ restricted to $V^{\lambda}$ is given by scalar multiplication by $\cont^\nu_{\Box_{\nu/\lambda}}$, we have 
\begin{equation}
\label{eq:TrJ1}
\mathrm{Tr}_{V^\nu} \mathcal{J}_{|\nu|} = 
\sum_{\lambda\in\nu\setminus 1} \dim\lambda \cdot \cont^\nu_{\Box_{\nu/\lambda}} = 
(|\nu|-1)!\sum_{\lambda\in\nu\setminus 1} 
\frac{\cont^\nu_{\Box_{\nu/\lambda}}} {H_\lambda}.
\end{equation}

On the other hand, $\mathcal{J}_{|\nu|}$ is just a sum of $|\nu|-1$ transpositions, and the trace of any transposition on $V^{\nu}$ is the same. Thus $\mathrm{Tr}_{V^\nu} \mathcal{J}_{|\nu|}$ is equal to $(2/|\nu|) \mathrm{Tr}_{V^\nu} (\mathcal{J}_2+\cdots+\mathcal{J}_{|\nu|})$. Note that operator $\mathcal{J}_2+\cdots+\mathcal{J}_{|\nu|}$ acts on $V^\nu$ simply by scalar multiplication by $\sum_{\Box\in\nu} \cont^\nu_\Box$. Thus we have:
\begin{equation}
\label{eq:TrJ2}
\mathrm{Tr}_{V^\nu} \mathcal{J}_{|\nu|} = 
\frac{2}{|\nu|} \dim\nu \sum_{\Box\in\nu} \cont^\nu_\Box= 
2\frac{(|\nu|-1)!}{H_\nu} \sum_{\Box\in\nu} \cont^\nu_\Box.
\end{equation}

Equating the right hand sides of the two expressions \eqref{eq:TrJ1}, \eqref{eq:TrJ2} for $\mathrm{Tr}_{V^\nu} \mathcal{J}_{|\nu|}$ gives the statement of the lemma.
\end{proof}

\begin{proof}[Combinatorial proof using Han's $g$-functions] The special case of $\nu=(1\geq1\geq\cdots\geq 1)$ can be checked directly. So we assume that $\ell(\nu) < |\nu|$ in the rest of the proof.
	
We use a result of Han, see~\cite[Theorem 1.1]{HanShifted}. Han defines a $g$-function of a partition $\nu$ as 
	\begin{equation*}
		g_\nu(x):=\prod_{i=1}^{|\nu|} (x+\nu_i-i).
	\end{equation*}
It is proved in \cite[Theorem 1.1]{HanShifted} that for any Young diagram $\nu$
	\begin{equation*}
 \sum_{\lambda\in\nu\setminus 1} \frac{g_\lambda(x)}{H_\lambda} = 		\frac{g_\nu(x+1)-g_\nu(x)}{H_\nu} .
	\end{equation*}
We rewrite this as 	
\begin{equation}\label{eq:HanTheorem}
\sum_{\lambda\in\nu\setminus 1} \frac{g_\lambda(x)}{g_\nu(x)} \frac 1 {H_\lambda}
=
\frac{\frac{g_\nu(x+1)}{g_\nu(x)}-1}{H_\nu} .
\end{equation}
and substitute $x=|\nu|/2+1/w$ with an intention to compare the coefficients of $w^3$ on both sides of this equation. This might seem arbitrary, but taking exactly these coefficients in the $w$-expansions of both sides of this equality produces precisely the result we claim in this Lemma, as shown below.

On the left hand side, let us assume that the coordinates of the box $\Box_{\nu/\lambda}$ are $(\nu_j,j)$. Then
\begin{equation*}
\frac{g_\lambda(x)}{g_\nu(x)} = \frac{
	(\textstyle\frac{|\nu|}2 +\textstyle\frac{1}{w} + \nu_j -j -1)
}{
	(\textstyle\frac{|\nu|}2 +\textstyle\frac{1}{w} + \nu_j -j )
		(\textstyle\frac{|\nu|}2 + \textstyle\frac{1}{w}  -|\nu|)
},
\end{equation*}
and the coefficient of $w^3$ is given by $(\nu_j-j)+|\nu|^2/4 $. Since  $\sum_{\lambda\in\nu\setminus 1} H_\lambda^{-1} = |\nu| H_\nu^{-1}$, the coefficient of $w^3$ on the left hand side of \cref{eq:HanTheorem} is equal to
\begin{equation}\label{eq:LHS}
 \frac{|\nu|^3}{4H_\nu} + \sum_{\lambda\in\nu\setminus 1} \frac{\cont^\nu_{\Box_{\nu/\lambda}}}{H_\lambda} .
\end{equation}

Now we compute the coefficient of $w^3$ on the right hand side of \cref{eq:HanTheorem}. We have:
\begin{equation*}
\frac{g_\nu(|\nu|/2+1/w+1)}{g_\nu(|\nu|/2+1/w)} = \prod_{i=1}^{|\nu|} (1+ w + w^2(i-\nu_i-\textstyle\frac{|\nu|}2)+w^3(i-\nu_i-\textstyle\frac{|\nu|}2)^2+\cdots),
\end{equation*}
and the coefficient of $w^3$ in this expression is equal to 
\begin{equation*}
\sum_{i=1}^{\ell(\nu)} \nu_i(\nu_i-2i+1) + \frac{|\nu|^3}4
= 2 \sum_{\Box\in\nu} \cont^\nu_\Box + \frac{|\nu|^3}4.
\end{equation*}
Thus the coefficient of $w^3$ on the right hand side of \cref{eq:HanTheorem} is equal to
\begin{equation}\label{eq:RHS}
\frac{|\nu|^3}{4H_\nu} + \frac{2}{H_\nu} \sum_{\Box\in\nu} \cont^\nu_\Box.
\end{equation}
Han's theorem, in the form of equation \eqref{eq:HanTheorem}, thus implies that expressions \eqref{eq:LHS} and \eqref{eq:RHS} must be equal, which leads to the statement of the lemma.
\end{proof}

\section{Topological recursion for monotone Hurwitz numbers}\label{sec:TR}

Topological recursion is a technique to reconstruct in a universal way solutions to various enumerative geometry problems using a small set of initial data. It was originally developed by Chekhov, Eynard, and Orantin
in the context
of matrix model theory, see~\cite{EynardOrantin}. In the case of the monotone Hurwitz numbers we specify the following input data: The curve is $\mathbb{C}\mathrm{P}^1$, the basic functions are $x=(z-1)/z^2$, $y=-z$, and the Bergman kernel is $B=dz_1 dz_2/ (z_1-z_2)^2$. The function $x$ has a unique critical point at $z=2$, and we denote by $\sigma$ the deck transformation in the neighborhood of this point. 
Using this data one reconstructs symmetric $n$-differentials $\omega_{g,n}(z_1,\dots,z_n)$ by the rule
\begin{align*}
\omega_{0,1} (z_1)\coloneq \ &  -y(z_1)dx(z_1); 
\\
\omega_{0,2} (z_1,z_2) \coloneq \ & B(z_1,z_2);
\\
\omega_{g,n+1} (z_0,z_{[n]}) \coloneq \ & \Res \limits_{z=2} \frac{\int_z^{\sigma(z)} B(z_0,\cdot)}{ydx(\sigma(z))-ydx(z)} \bigg[
\omega_{g,n+2}(z,\sigma(z),z_{[n]})  \\
& + \sum'_{\substack{
h+k=g \\
I\sqcup J = [n]
}}
\omega_{h,|I|+1}(z,z_I)
\omega_{k,|J|+1}(\sigma(z),z_J)
\bigg]
\end{align*}
Here by $\sum'$ we mean that we take the sum excluding the cases when $(h,|I|+1)$ or $(k,|J|+1)$ are equal to $(0,1)$, and by $z_S$ for any set $S\subset [n]=\{1,\dots,n\}$ we denote the tuple of variables indexed by $S$.

The connected monotone Hurwitz numbers $H^\circ_{g,\mu}$ are defined in terms of the disconnected monotone Hurwitz numbers via the inclusion-exclusion formula. The generating function is given by
\begin{equation*}
\log\mathcal{Z}(s,t,\mathbf{p}) \eqcolon \sum_{\substack{d\geq 1,\ \mu\vdash d,\ b\geq 0 \\ g:=\frac{b-d-\ell(\mu))}2 \in \mathbb{Z}_{\geq 0}}
	}
\frac{s^d}{d!} t^{b}  H^\circ_{g,\mu} p_{\mu_1}\cdots p_{\mu_{\ell(\mu)}}
\end{equation*}
Denote by $F_{g,n}$ the $n$-point function $\sum_{\ell(\mu)=n}   H^\circ_{g,\mu} \prod_{i=1}^n x_i^{\mu_i}$.

We give a new proof of the following theorem of Do, Dyer, and Mathews that connects the formal expansion of differentials $\omega_{g,n}$ in the variable $x$ and the $n$-point functions.

\begin{theorem}[\cite{DoDyerMathews}] We have: $\omega_{0,1} = d F_{0,1}$, $\omega_{0,2} = d_1d_2F_{0,2} + dx_1dx_2/(x_1-x_2)^2$, and 
$\omega_{g,n} = d_1\cdots d_n F_{g,n}$ for $2g-2+n>0$.
\end{theorem}

\begin{proof}
	First, one has to check by hand the cases of $(g,n)=(0,1)$ and $(0,2)$. It is done in~\cite[Lemma 9 and Proposition 14]{DoDyerMathews}, see also~\cite[Theorems A.1 and A.3]{KLS}. Then, one has to check that the formal power series $F_{g,n}$ is indeed an expansion of the products of certain functions on the curve, that is, the multiple $\partial/\partial x$ derivatives of $1/(z-2)$. This is equivalent to the polynomiality property proved in~\cite[Theorem 1.5]{GouldenPolynomiality}, see also~\cite[Theorem 5.2]{KLS} (note that this property is highly non-trivial). Once these preliminary steps are completed, the topological recursion is equivalent to the quadratic loop equation, see~\cite[Theorem 2.2]{BorotShadrin}. Namely, we have to prove that the functions 
	\begin{equation*}
	DF_{g,n}(z_{[n]}) \coloneq \frac{d_1\cdots d_n F_{g,n}}{dx_1\cdots dx_n}
	\end{equation*}
	(note that the polynomiality property allows us to consider these expressions not as formal power series, but as global functions defined on $\mathbb{C}\mathrm{P}^1$) satisfy the following property:
	\begin{equation*}
	DF_{g,n+2}(z,\sigma(z),z_{[n]})  + \sum_{\substack{
			h+k=g \\
			I\sqcup J = [n]
	}}
	DF_{h,|I|+1}(z,z_I)
	DF_{k,|J|+1}(\sigma(z),z_J)
	\end{equation*}
	is holomorphic in $z$ in the neighborhood of $z=2$.
	In order to prove this, we use the cut-and-join equation proved in the previous section. Namely, consider \cref{eq:CJ} specialized for the \emph{connected} monotone Hurwitz numbers:
	\begin{align*}
	& \frac{1}{2t} \left[s\frac{\partial \log\mathcal{Z}}{\partial s} - sp_1\right] =
	\\ \notag &
	\frac 12
	\sum_{i,j=1}^\infty 
	\left[
	(i+j)p_ip_j \frac{\partial\log\mathcal{Z}} {\partial p_{i+j}} + ij p_{i+j} \frac{\partial^2 \log\mathcal{Z}}{\partial p_i\partial p_j}
	+ ij p_{i+j} \frac{\partial \log\mathcal{Z}}{\partial p_i} \frac{\partial \log\mathcal{Z}}{\partial p_j}
	\right]
	.
	\end{align*}
	As shown in~\cite{DBKOSS, DLPS}, under the condition $\omega_{0,2} = d_1d_2F_{0,2} + dx_1dx_2/(x_1-x_2)^2$ and the polynomiality property (which are both satisfied in our case) this equation, rewritten in terms of $n$-point functions and symmetrized with respect to the deck transformation, gives the quadratic loop equation. Note that the left hand side of the cut-and-join equation is different in this case and in~\cite{DBKOSS, DLPS}, but symmetrization with respect to the deck transformation makes these contributions holomorphic in any case.  
	This completes the proof of the theorem.
\end{proof}

\begin{remark} Since this method of proof of the topological recursion was already used in the literature (cf.~\cite{DLPS}), it probably falls into the ``known to the experts'' category in this case. We give here its brief account merely for completeness, as the illustration how closely the second Casimir operator $C_2$ (or, equivalently, $\mathcal{F}_2$ in the semi-infinite wedge formalism) is related to the quadratic loop equation. In fact, modulo some local expansion analysis and very general statements on topological recursion, the combination of two theorems of Goulden, Guay-Paquet, and Novak, \cite[Theorem 1.5]{GouldenPolynomiality} and \cite[Theorem 1.2]{GouldenetalGenus0} immediately implies the topological recursion statement.
\end{remark}


\bibliographystyle{alpha}

\bibliography{monotonerev}

\end{document}